\documentclass[12pt,twoside]{article}

\usepackage{enumerate}
\usepackage[noadjust]{cite}
\usepackage{amssymb,amsmath,amsthm}
\usepackage{comment}
\usepackage{color}
\newcommand\blue{\color{blue}}

\newcommand\mytitle{Boundary systems
and (skew-)self-adjoint operators on infinite metric graphs}

\swapnumbers
\numberwithin{equation}{section}

\newtheorem{theorem}{Theorem}[section]
\newtheorem{corollary}[theorem]{Corollary}
\newtheorem{proposition}[theorem]{Proposition}
\newtheorem{lemma}[theorem]{Lemma}
\theoremstyle{definition}
\newtheorem*{definition}{Definition}
\newtheorem{remark}[theorem]{Remark}
\newtheorem{remarks}[theorem]{Remarks}

\newtheorem{examples}[theorem]{Examples}

 \mathchardef\ordinarycolon\mathcode`\:
  \mathcode`\:=\string"8000
  \begingroup \catcode`\:=\active
    \gdef:{\mathrel{\mathop\ordinarycolon}}
  \endgroup

\newcommand\smid{\nonscript \mskip2mu plus2mu {\mid}%
\nonscript \mskip2mu plus2mu}     
\def\scpr(#1,#2){{(#1\smid#2)}}
\def\bigscpr(#1,#2){{\left(#1\nonscript \mskip2mu plus2mu \middle| \nonscript \mskip2mu
plus2mu#2\right)}}

\newcommand\vtx{\gamma}

\newcommand\bm{F}

\newcommand\trace{\operatorname{tr}}
\newcommand\strace{\operatorname{str}}
\newcommand\li{{\rm l}}
\newcommand\re{{\rm r}}

\newcommand{\spt}{\operatorname{spt}}

\newcommand{\lin}{\operatorname{lin}}

\newcommand\imu{{\rm i}}
\renewcommand\phi{\varphi}

\renewcommand\epsilon{\varepsilon}

\newcommand{\R}{\mathbb{R}\nonscript\hskip.03em}
\newcommand{\N}{\mathbb{N}\nonscript\hskip.03em}
\newcommand{\Z}{\mathbb{Z}\nonscript\hskip.03em}
\newcommand{\C}{\mathbb{C}\nonscript\hskip.03em}
\newcommand{\K}{\mathbb{K}\nonscript\hskip.03em}

\newcommand\rma{{\rm (a) }}
\newcommand\rmb{{\rm (b) }}
\newcommand\rmc{{\rm (c) }}
\newcommand\rmd{{\rm (d) }}
\newcommand\gr[1]{#1}
\renewcommand\gr[1]{G(#1)}
\newcommand\fin{{\rm fin}}

\def\formE(#1,#2){\sum_{e\in E}\int_{a_e}^{b_e} #1_e'(x)\ol{#2_e'(x)}\,dx}

\makeatletter
\let\qedhere@ams\qedhere
\def\qedhere{\@ifnextchar[{\@qedhere}{\qedhere@ams}}
\def\@qedhere[#1]{\tag*{\raisebox{-#1ex}{\qedhere@ams}}}
\makeatletter
\def\env@cases{%
  \let\@ifnextchar\new@ifnextchar
  \left\lbrace
  \def\arraystretch{1.1}%
  \array{@{\,}l@{\quad}l@{}}%
}
\renewcommand\section{\@startsection {section}{1}{\z@}%
                                     {-3.25ex \@plus -1ex \@minus -.2ex}%
                                     {1.5ex \@plus.2ex}%
                                     {\normalfont\large\bfseries}}
\makeatother

\newcommand\set[2]{\bigl\{#1{;}\;#2\bigr\}}
\newcommand\sset[2]{\{#1{;}\;#2\}}
\newcommand\ol{\overline}

\renewcommand\le{\leqslant}

\renewcommand\ge{\geqslant}

\newcommand{\from}{{:}\;}
\renewcommand{\from}{\colon}

\newcommand\sse{\subseteq}

\newcommand\cH{\mathcal{H}}
\newcommand\cG{\mathcal{G}}

\newcommand\cHg{{\mathcal H_\Gamma}}

\newcommand\pskew{\hbox{(skew-)}\hskip0pt}

\newcommand{\abstracttext}{\noindent
We generalize the notion of Lagrangian subspaces to self-orthogonal subspaces
with respect to a
\pskew symmetric form, thus characterizing \pskew self-adjoint
and unitary operators by means of self-ortho\-gonal subspaces.
By orthogonality preserving mappings, these characterizations can be
transferred to abstract boundary value spaces of \pskew symmetric operators.
Introducing the notion of boundary systems we then present a unified treatment
of different versions of boundary triples and related concepts treated in the
literature.
The application of the abstract results yields a
description of all \pskew self-adjoint realizations of
Laplace and first derivative operators on graphs.

\vspace{8pt}

\noindent
MSC 2010: 47B25, 35Q99, 05C99
\vspace{2pt}

\noindent
Keywords: boundary triple, \pskew self-adjoint operators, quantum graphs
}
\begin{document}
\title{\mytitle}

\author{Carsten Schubert, Christian Seifert, J\"urgen Voigt\\ and Marcus
Waurick}

\date{}

\maketitle

\begin{abstract}
\abstracttext
\end{abstract}

\section*{Introduction}

The problem of obtaining \pskew self-adjoint extensions of \pskew symmetric
operators is a standard task in operator theory and applications. The classical
description, given by von Neumann, is by finding unitary maps between
deficiency spaces.
A seemingly different established description is to find all Lagrangian
subspaces with respect to the corresponding boundary form{\blue ;} see
e.g.~\cite{brge03,brgepa2008,es10,har00,gor91,cod73}. This method has been
formalized in the language of boundary triples.
For the case of \pskew symmetric operators a procedure has been described in
\cite{waukal11}, using the notion of SWIPs
(systems with integration by parts).

It is one of the purposes of our paper to present a unified treatment of
the methods mentioned above. This is achieved by introducing the notion of
`boundary systems', a generalized version of boundary triples. The
generalization consists in allowing more flexibility for the quantities
occurring in the setup.

Self-adjoint Laplace
operators on metric graphs arise by choosing appropriate boundary conditions at
the vertices.
The first treatment of this topic was given in 
\cite{ks99}, characterizing self-adjointness of Laplacians on finite star
graphs
by Lagrangian subspaces with respect to a standard form.
Kuchment \cite{kuc04} gave another description of the boundary conditions
leading to semibounded self-adjoint Laplacians on graphs with finite vertex
degree. The question
whether all self-adjoint realizations of the Laplacian on a metric
graph can be obtained by choosing self-adjoint operators in the space of
boundary values arose in the thesis \cite{schubert2011} of one of the 
authors.
With the help of our methods we answer this question for graphs with a positive
lower bound for the edge lengths in Theorem \ref{s-a Lapl}.
This is essentially the description of Laplacians on graphs in the form 
treated by 
Kuchment \cite{kuc04}. 
We stress that we can deal with infinite metric graphs with 
infinite vertex degree which have 
been studied recently \cite{schubert2011,lsv2012}.
Since we also do not assume semi-boundedness of the corresponding 
operator, the form approach may not be applicable.
\medskip

In Section \ref{sec:sls} 
we introduce the abstract context, and we describe the relation of
self-orthogonal subspaces to \pskew self-adjoint or unitary operators.

In Section \ref{sec:bs} we introduce boundary systems, and we present the
abstract results how \pskew self-adjointness of extensions of \pskew symmetric
operators can be described by self-orthogonal subspaces in the `boundary
space'. We mention that in the hypotheses of a boundary system there is no
a priori requirement concerning the deficiency indices of the operator to be
extended.

In Section~\ref{sec:qg} we
apply the abstract theory to Laplace operators and the first derivative operator
on metric 
graphs.
On a metric graph, the Laplace operator is self-adjoint if and only if
all the boundary values of a function in the domain are related to all boundary
values of the derivative of the function via some self-adjoint operator acting
in a subspace of all possible boundary values. 
In two examples we present the application of our result to infinite 
graphs, one of them with infinite vertex degree.
For the derivative operator, we
have that the respective operator is skew-self-adjoint if and only if the
relation between the boundary values at the end points of the edges and the
boundary values at the starting points of the edges is unitary.

\section{Sesquilinear forms and self-orthogonal subspaces}
\label{sec:sls}

We start with basic observations concerning a version of `orthogonality'. Let 
$X$ be a set, and let $R\sse X\times X$ be an `orthogonality relation'. For 
$U\sse X$ we define the \emph{$R$-orthogonal complement}
\[
U^{\perp_R}:=\set{x\in X}{(x,y)\in R\ (y\in U)}
\]
of $U$, and $U$ will be called \emph{$R$-self-orthogonal} if $U^{\perp_R}=U$.

\begin{theorem}\label{map-so}
Let $X_1,X_2$ be sets, and let $R_j\sse X_j\times X_j$ ($j\in\{1,2\}$). Let 
$F\colon X_1\to X_2$ be surjective, and assume that
\[
R_1 = (\bm\times\bm)^{-1}(R_2).
\]

Then a set $U\sse X_1$ is $R_1$-self-orthogonal if and only if there exists 
an $R_2$-self-orthogonal set $V\sse X_2$ such that $U=\bm^{-1}(V)$.
\end{theorem}

\begin{proof}
Taking into account the surjectivity of $F$, one easily obtains that
\[
\bm^{-1}(V)^{\perp_{R_1}}=\bm^{-1}(V^{\perp_{R_2}})
\]
for all $V\sse X_2$. Also, one obtains that
\[
U^{\perp_{R_1}}=F^{-1}(\bm(U))^{\perp_{R_1}}\, 
(\,=\bm^{-1}(F(U)^{\perp_{R_2}})\,)
\]
for all $U\sse X_1$.

From these observations the assertions of the theorem are immediate.
\end{proof}

We now introduce \pskew symmetric forms and self-orthogonal subspaces.
All vector spaces will be vector spaces over $\K$, where $\K\in\{\R,\C\}$.

\begin{definition}
Let $X$ be a vector space,
and let
$\omega\from X\times X\to \K$ be sesquilinear.
We use the relation
\[
R:=\set{(x,y)\in X\times X}{\omega(x,y)=0}
\]
in order to define `orthogonality' in $X$, and we thus write
\[
U^{\perp_\omega}:=U^{\perp_R}=\set{x\in X}{\omega(x,y) = 0 \ (y\in U)},
\]
for $U\sse X$.
%
%
We will use the terminology \emph{$\omega$-self-orthogonal}
to mean orthogonality with respect to the above relation $R$. 

The form $\omega$ is called \emph{symmetric}, if
\[\omega(x,y) = \overline{\omega(y,x)} \quad(x,y\in X),\]
and \emph{skew-symmetric} (or \emph{symplectic}), if
\[\omega(x,y) = -\overline{\omega(y,x)} \quad(x,y\in X).\]
\end{definition}

In the case of skew-symmetric forms, self-orthogonal subspaces are also called
\emph{Lagrangian}.
Note that there are different definitions of Lagrangian subspaces in the
literature; see e.g. \cite{evma04,har00,kps08}.

The following result is a reformulation of Theorem~\ref{map-so} for the present 
context.

\begin{corollary}\label{prop:surj}
Let $X_1, X_2$ be vector spaces, and let $\omega_j$ be a sesquilinear form
on $X_j$ ($j\in\{1,2\}$). Let $\bm\from X_1\to X_2$ be linear and
surjective, and such that
\[\omega_1(x,y) = \omega_2(\bm(x),\bm(y)) \quad(x,y\in X_1 ).\]

Then a subspace $U\sse X_1$ is $\omega_1$-self-orthogonal if and only if there 
exists 
an $\omega_2$-self-orthogonal subspace $V\sse X_2$ such that $U=\bm^{-1}(V)$.
\end{corollary}

\begin{remark}\label{rem-interpret}
In the context of the preceeding corollary, one should think of $X_1$
as a (big) space of functions, of $X_2$ as the (small) space of boundary
values, and of the mapping $\bm$ as the evaluation of the boundary values of the
functions in $X_1$.
\end{remark}

\begin{definition}
  Let $\cH_1,\cH_2$ be Hilbert spaces.

\rma
A subspace $M\subseteq \cH_1\oplus\cH_2$ is called a \emph{linear relation}.
  For a linear relation $M\subseteq \cH_1\oplus\cH_2$ we define the
\emph{inverse relation} $M^{-1}\sse\cH_2\oplus\cH_1$ by
\[
M^{-1} := \set{(y,x)\in \cH_2\oplus\cH_1}{(x,y)\in M},
\]
the \emph{orthogonal relation} of $M$ by
\[
M^\perp := \set{(x,y)\in
\cH_1\oplus\cH_2}{\scpr({(x,y)},{(u,v)})_{\cH_1\oplus\cH_2} = 0\
((u,v)\in M)},
\]
and the \emph{adjoint relation} $M^*\sse\cH_2\oplus\cH_1$ by
\[
M^*:=\set{(y,x)\in\cH_2\oplus\cH_1}
{\scpr(y,v)_{\cH_2}=\scpr(x,u)_{\cH_1}\ ((u,v)\in M)}.
\]

\rmb
If $\cH_1 = \cH_2$ and $M\subseteq M^*$, then $M$ is called \emph{symmetric},
and if $M=M^*$, then $M$ is called \emph{self-adjoint}.

\rmc
Denote $S:= \left(\begin{smallmatrix} 1 & 0 \\ 0 & -1 \end{smallmatrix}\right)$.
If $\cH_1 = \cH_2$ and $M\subseteq SM^*$, then $M$ is called
\emph{skew-symmetric}, and if $M=SM^*$, then $M$ is called
\emph{skew-self-adjoint}.

\rmd
If $M^* = M^{-1}$, then $M$ is called \emph{unitary}. In fact, unitary
relations are graphs of unitary operators; see Proposition \ref{prop:uni} below.
\end{definition}

\begin{remark}
Let $\cH_1,\cH_2$ be Hilbert spaces, $M\subseteq \cH_1\oplus\cH_2$ a linear relation.
Then
\[M^*= (SM^\perp)^{-1} = ((SM)^\perp)^{-1} = ((SM)^{-1})^\perp =
(SM^{-1})^\perp = S(M^{-1})^\perp.\]
\end{remark}

\begin{examples}
\label{ex:forms}
\rma Let $\cH$ be a Hilbert space. The \emph{standard skew-symmetric}
(or \emph{symplectic})
\emph{form} on $\cH\oplus\cH$ is defined as $\omega\from
(\cH\oplus\cH)\times (\cH\oplus\cH)\to \K$,
\[\omega((x,y),(u,v)) = \bigscpr(\begin{pmatrix} 0 & -1 \\ 1 & 0 \end{pmatrix}
\begin{pmatrix} x\\y \end{pmatrix} , \begin{pmatrix} u\\v
\end{pmatrix})_{\cH\oplus\cH} = \scpr(x , v)_{\cH}-\scpr(y , u)_{\cH}.\]
Let $M\subseteq \cH\oplus\cH$ be a linear relation.
Then
\begin{align}\label{ortho-id}
M^{\perp_\omega} = ((SM)^{-1})^\perp = M^*.
\end{align}
Hence, Lagrangian subspaces with respect to the standard skew-symmetric form are exactly the self-adjoint linear relations.

\rmb Let $\cH$ be a Hilbert space. We define the \emph{standard symmetric
form} on $\cH\oplus\cH$ by $\omega\from
(\cH\oplus\cH)\times (\cH\oplus\cH)\to \K$,
\[\omega((x,y),(u,v)) = \bigscpr(\begin{pmatrix} 0 & 1\\ 1 & 0 \end{pmatrix}
\begin{pmatrix} x\\y \end{pmatrix} , \begin{pmatrix} u\\v
\end{pmatrix})_{\cH\oplus\cH} = \scpr(x , v)_{\cH} + \scpr(y , u)_{\cH}.\]

Let $M\subseteq \cH\oplus\cH$ be a linear relation. Then
\begin{align}\label{ortho-id-sym}
M^{\perp_\omega} = (M^{-1})^\perp = SM^*.
\end{align}
Hence, self-orthogonal subspaces with respect to the standard symmetric form are
exactly the skew-self-adjoint linear relations.

\rmc Let $\cH_1,\cH_2$ be Hilbert spaces. We define the \emph{standard
unitary form} on $\cH_1\oplus\cH_2$ by $\omega\from
(\cH_1\oplus\cH_2)\times (\cH_1\oplus\cH_2)\to \K$,
\[\omega((x,y),(u,v)) = \bigscpr(\begin{pmatrix} 1 & 0 \\ 0 & -1 \end{pmatrix}
\begin{pmatrix} x\\y \end{pmatrix} , \begin{pmatrix} u\\v
\end{pmatrix})_{\cH_1\oplus\cH_2} = \scpr(x , u)_{\cH_1} - \scpr(y ,
v)_{\cH_2}.\]

Let $M\subseteq \cH_1\oplus\cH_2$ be a linear relation.
Then
\begin{align}\label{ortho-id-unit}
M^{\perp_\omega} = (SM)^\perp =  (M^*)^{-1}.
\end{align}
Hence, the self-orthogonal subspaces with respect to
the standard unitary form are exactly the unitary relations.
\end{examples}

We now describe the self-orthogonal subspaces we are dealing with in the applications.

\begin{proposition}[{cf. \cite[Theorem 5.3]{arens61}}]
\label{thm:lag_XL}
Let $\cH$ be a Hilbert space. Then a linear relation $U\sse \cH\oplus \cH$ is
self-adjoint (i.e., Lagrangian in $\cH\oplus\cH$ with respect to
the standard skew-symmetric form) if and only if there exist
a closed linear subspace
$X\subseteq \cH$ and a self-adjoint operator $L$ in $X$, such that $U = G(L)
\oplus (\{0\} \oplus X^\perp)$, where $G(L)=\set{(x,Lx)\in X\oplus X}{x\in
D(L)}$ denotes the graph of $L$.
\end{proposition}

\begin{proof}
We only give a
short outline of the ideas. For a more detailed proof we refer to
\cite[Section 5]{arens61}.

If $X$ and $L$ are as indicated, then the space $\gr L$ is Lagrangian in
$X\oplus X$ and clearly, $\{0\} \oplus X^\perp$ is a Lagrangian subspace
of $X^\perp\oplus X^\perp$. It is not difficult to show that this implies that
$U = G(L) \oplus (\{0\} \oplus X^\perp)$ is Lagrangian.

Assume that $U$ is Lagrangian, and let $\cH_\infty:= \set{y\in\cH}{(0,y)\in U}$,
$X:= \cH_\infty^\perp$. Then
one shows that $U\cap(X\oplus X)$ is the graph of
an operator $L$, and that then $U$ is of the described form.
\end{proof}

\begin{remarks}
\label{rem:skew_XL}
(a) It follows that a self-adjoint linear relation $U\sse\cH\oplus\cH$ is the
graph of an operator if and only if its domain $P_1U$ ($P_1$ the projection
onto the first component) is dense in $\cH$.

(b) A result analogous to Proposition \ref{thm:lag_XL} holds for
self-orthogonal subspaces with respect to the standard symmetric form. Then $L$
will be a skew-self-adjoint operator.
\end{remarks}

\begin{proposition}
\label{prop:uni}
Let $\cH_1,\cH_2$ be Hilbert spaces, $U\subseteq \cH_1\oplus \cH_2$ a subspace.
Then $U$ is the graph of a unitary operator $L$ if and only if $U =
U^{\perp_\omega}$, where $\omega$ is the standard unitary form on
$\cH_1\oplus\cH_2$.
\end{proposition}

\begin{proof}
Let $U = U^{\perp_\omega}$. Then
\begin{align}\label{formel2}
0 = \omega((x,y),(x,y)) = (\|x\|_{\cH_1}^2-\|y\|_{\cH_2}^2)\quad((x,y)\in U),
\end{align}
and this implies that $U$ is the graph of a (closed) isometric operator $L$.
Let $u\in D(L)^\perp$. Then $(u,0)\in U^{\perp_\omega}=U$, and \eqref{formel2} implies that $u=0$, i.e., $D(L)$ is
dense. Similarly one obtains that the range of $L$ is dense. This implies that
$L$ is unitary.

Let $L\from\cH_1\to\cH_2$ be unitary. Then $\gr{L}^* = \gr{L^*} = \gr{L^{-1}} =
\gr{L}^{-1}$ and therefore $\gr{L}^{\perp_\omega} = \gr{L}$.
\end{proof}

The following correspondence between skew-self-adjoint relations and unitary
operators will establish the link between our setup and one of the versions of
boundary triples; cf. Remark~\ref{bound-tr-unitary}.
We define the unitary mapping $C$ in $\cH\oplus\cH$, given by the matrix $C:=
\frac1{\sqrt2}
\left(
\begin{smallmatrix}
1 &1\\
{-1} &1
\end{smallmatrix}
\right)$. Let $\omega_{\rm s}$ be the standard symmetric form and
$\omega_{\rm u}$ the standard unitary form on $\cH\oplus\cH$. One checks that
then
\[
\omega_{\rm s}((x,y),(u,v))=
\omega_{\rm u}(C(x,y),C(u,v))\qquad((x,y),(u,v)\in\cH\oplus\cH).
\]

\begin{proposition}\label{skew-sa-unitary}
Let $C$ be as above. Then a linear relation $U\sse\cH\oplus\cH$ is
skew-self-adjoint if and only if $CU$ is the graph of a unitary operator.
\end{proposition}

\begin{proof}
Theorem \ref{prop:surj}(b) implies that the $\omega_{\rm s}$-self-orthogonality of
$U$ is equivalent to the $\omega_{\rm u}$-self-orthogonality of $CU$. Applying
Remark~\ref{rem:skew_XL}(b) and Proposition~\ref{prop:uni} one obtains the
assertion.
\end{proof}

\begin{remarks}
(a) If $A$ is a skew-self-adjoint operator in $\cH$, then $C$ applied to the
graph of $A$ yields the graph of the Cayley transform $(A-I)(A+I)^{-1}$ of $A$.

(b) The statement corresponding to Proposition~\ref{skew-sa-unitary}, for
self-adjoint operators instead of skew-self-adjoint operators requires a
complex Hilbert space. In this form, the result is contained in
\cite[Theorem~4.6]{arens61}. The mapping inducing the equivalence is then
defined by the matrix $C:=
\frac1{\sqrt2}
\left(
\begin{smallmatrix}
1 &-\imu\\
{-1} &-\imu
\end{smallmatrix}
\right)$, for the Cayley transform $(A-\imu)(A+\imu)^{-1}$ of a self-adjoint
operator $A$.
\end{remarks}

\section{Boundary systems}
\label{sec:bs}

Let $\cH$ be a Hilbert space, $H_0$ a symmetric or skew-symmetric operator in
$\cH$.

\begin{definition}
  A \emph{boundary system} $(\Omega,\cG_1,\cG_2,\bm,\omega)$ for $H_0$ consists
of a sesquilinear form $\Omega$ on $G(H_0^*)$,
two Hilbert spaces $\cG_1,\cG_2$, a linear and surjective mapping $\bm\from 
G(H_0^*)\to \cG_1\oplus\cG_2$ and a sesquilinear form $\omega$ on 
$\cG_1\oplus\cG_2$, such that 
\[
\Omega((x,H_0^*x),(y,H_0^*y)) = \omega(\bm(x,H_0^*x),\bm(y,H_0^*y))
\quad(x,y\in D(H_0^*)).
\]
\end{definition}

For a boundary system let $F_j\from D(H_0^*)\to \cG_j$ ($j\in\{1,2\}$) such that 
$F(x,H_0^*x) = (F_1(x),F_2(x))$ for all $x\in D(H_0^*)$. 

\begin{remarks}
(a)
If $\cG_1 = \cG_2 = \cG$ and both $\Omega$ and $\omega$ are the standard
skew-symmetric forms (on the corresponding spaces), then
$(\Omega,\cG,\cG,\bm,\omega)$ corresponds to the version of a
\emph{boundary triple} as treated in \cite{brge03, brgepa2008}. The usual
notation is $(\cG,\Gamma_1,\Gamma_2)$, where
$\Gamma_1,\Gamma_2\from D(H_0^*)\to \cG$ are such that
$(\Gamma_1,\Gamma_2)\from D(H_0^*)\to \cG\oplus\cG$ is surjective. Note that 
here $\Gamma_j = F_j$ ($j\in\{1,2\}$). 

(b)
If $\Omega$ is the standard skew-symmetric form on $G(H_0^*)$, 
then $\Gamma$, given by $\Gamma(x,y):= \Omega((x,H_0^*x),(y,H_0^*y))$ ($x,y\in
D(H_0^*)$) is also called the
\emph{boundary form} of $H_0$; cf.~\cite[Def.\,7.1.1]{deOliv2009}.

(c)
The notion of \emph{systems with integration by parts (SWIPs)}, defined
in \cite[Definition 3.4]{waukal11} deals with a skew-symmetric operator $H_0$.
Then SWIPs correspond to the case that $\Omega$ and $\omega$ are the standard
skew-symmetric forms on $\cH \oplus \cH$ and $\cG \oplus \cG$, respectively,
with $\cG=\cG_1=\cG_2$. Additionally, $F\colon G(H_0^*)\to \cG \oplus \cG$ is
assumed to be continuous (with respect to the graph norm of
$H_0^*$), to vanish on $G(H_0)$, and to be bijective on $G(H_0)^{\bot}\cap
G(H_0^*)$.
\end{remarks}

The following theorem is a version of the well-established result how
self-adjoint
extensions of symmetric operators can be obtained using boundary triples; see
\cite{brge03, brgepa2008, deOliv2009} and references therein. In our context
the domains of the self-adjoint extensions are characterized by
means of self-adjoint relations in the space of boundary values.

\begin{theorem}
\label{thm:sa_XL}
Let $H_0$ be a symmetric operator, and
let $(\Omega,\cG,\cG,\bm,\omega)$ be a boundary system for $H_0$, where
$\Omega,\omega$ are the standard skew-symmetric forms.
Then an operator $H\sse H_0^*$ is self-adjoint if and only if there 
exist a closed subspace $X\subseteq \cG$ and a
self-adjoint operator $L$ in $X$ such that
\[
D(H) = \set{x\in D(H_0^*)}{\bm_1(x) \in D(L),\, L\bm_1(x) =
Q\bm_2(x)},
\]
where $Q\from \cG\to X$ is the orthogonal projection.
\end{theorem}

In the proof we will need the following auxiliary result:

\begin{lemma}\label{sub-symm-adj}
Let $H_0$ be a symmetric operator, and let $\Omega$ be the standard
skew-symmetric form on $\cH\oplus\cH$ and $\check\Omega$ its restriction to
$G(H_0^*)$. Then an operator $H\sse H_0^*$ is self-adjoint if and only
if its graph is a Lagrangian subspace of $G(H_0^*)$ with respect to
$\check\Omega$.
\end{lemma}

\begin{proof}
Necessity is trivial. In order to show sufficiency let $U\sse G(H_0^*)$ be
Lagrangian in $G(H_0^*)$. Then $U^{\perp_\Omega}\supseteq
G(H_0^*)^{\perp_\Omega}=(G(H_0)^{\perp_\Omega})^{\perp_\Omega}
\supseteq G(H_0)$, and therefore $U=U^{\perp_{\check\Omega}}
= U^{\perp_\Omega}\cap G(H_0^*)\supseteq G(H_0)$.
This implies that $U^{\perp_\Omega}\sse G(H_0^*)$, and finally
that $U^{\perp_\Omega} = U^{\perp_\Omega}\cap G(H_0^*)=U^{\perp_{\check\Omega}}
=U$.
\end{proof}

\begin{proof}[Proof of Theorem~\ref{thm:sa_XL}]
  By Lemma \ref{sub-symm-adj}, $H$ is self-adjoint if and only if $G(H)$
is $\Omega$-self-orthogonal in $G(H_0^*$). By Corollary 
\ref{prop:surj}, 
  this is equivalent to the existence of an $\omega$-self-orthogonal 
subspace $V\sse\cG\times\cG$ such that $G(H)=\bm^{-1}(V)$. The description of 
$V$ in Proposition~\ref{thm:lag_XL} then yields the description of $H$ as a 
restriction of $H_0^*$.
\end{proof}

\begin{remark}\label{sub-skew-symm-adj}
The corresponding statement for skew-symmetric and skew-self-adjoint
operators is shown in the same way.
\end{remark}

Boundary systems as in the previous theorem make use of standard skew-symmetric
forms in the (big) Hilbert
space $\cH$ as well as in the (small) Hilbert space $\cG$, thus
characterizing self-adjoint restrictions of $H_0^*$ by means of Lagrangian
subspaces in $\cG\oplus\cG$. 
In the following statement we describe the application of boundary 
systems to the characterization of skew-self-adjoint extensions of 
skew-symmetric operators by unitary operators in the `boundary space'.

\begin{theorem}
\label{thm:ssa_L}
Let $H_0$ be a skew-symmetric operator, and 
let $(\Omega,\cG_1,\cG_2,\bm,\omega)$ be a boundary system for $H_0$,
where $\Omega$ is the standard symmetric form and $\omega$ is the
standard unitary form.
Then the operator $H\sse H_0^*$ is skew-self-adjoint if and only if
there exists a unitary operator $L$ from $\cG_1$ to
$\cG_2$ such that
\[
          D(H) = \sset{ x\in D(H_0^*)}{ L \bm_1(x) = \bm_2(x) }.
\]
\end{theorem}

\begin{proof}
 By Remark \ref{rem:skew_XL}(b), skew-self-adjoint operators correspond 
to self-ortho\-gonal subspaces with respect to the standard symmetric form. 
 By Corollary \ref{prop:surj}, $\bm$ establishes a 
correspondence between these subspaces and self-orthogonal subspaces 
with respect to the standard unitary form.
According to Proposition \ref{prop:uni} these subspaces are just the graphs of unitary operators.
\end{proof}

\begin{remarks}\label{bound-tr-unitary}
(a)
Assuming that $\cH$ is a complex Hilbert space, that $H_0$ is symmetric, and
letting $\Omega$ be the standard skew-symmetric form multiplied by the
imaginary unit, the (otherwise unchanged) setup of
Theorem~\ref{thm:ssa_L} yields the version of boundary triples as treated in
\cite{gor91,deOliv2009}. (We note that the hypothesis in
\cite[Section 7.1.1]{deOliv2009} that $\rho_1,\rho_2$ have dense range should
be replaced by the requirement that
$\rho_1\times\rho_2$
is surjective.) The
equivalence of the two versions of boundary triples is clear from
Proposition~\ref{skew-sa-unitary}.

(b)
We mention that in our setup there is no need for the a priori requirement of
equality of deficiency indices. 
\end{remarks}

\section{Application to quantum graphs}\label{sec:qg} 
 
We now apply the 
results 
of Section 
\ref{sec:bs} 
to Laplacians and first
derivative operators on metric graphs. In order to do so we
first introduce the relevant notions.
The following description of (directed multi-) graphs is an extension of the
notation presented in \cite{kkvw09}.

Let $\Gamma = (V,E,a,b,\vtx_0,\vtx_1)$ be a metric graph. 
This means that $V$ is the set of \emph{vertices} (or \emph{nodes}) of $\Gamma$ and
$E$ the set of \emph{edges}.
Furthermore let $a,b\from E\to [-\infty,\infty]$, and assume that $a_e < b_e$
and that the interval $(a_e,b_e)\sse \R$ corresponds to the edge $e$ ($e\in E$).
Denote $E_\li := \sset{ e \in E}{ a_e>-\infty }$ and $E_\re := \sset{ e\in E }{ b_e
<\infty }$.
Let $\vtx_0\from E_\li\to V$, $\vtx_1\from E_\re\to V$ associate with each edge 
$e\in E_\li$ or $e\in E_\re$, respectively, a ``starting vertex'' $\vtx_0(e)$ or 
an ``end vertex'' $\vtx_1(e)$, respectively.
Note that we do not assume finiteness (or countability) of
$V$ and $E$.

Assume that there is a positive lower bound for the lengths of the edges,
i.e.,
\begin{equation}
\label{eq:edgelength}
l:=\inf_{e\in E} (b_e - a_e) >0.
\end{equation}
The self-adjoint operators we treat will act in the Hilbert space
\[
\cHg:=\bigoplus_{e\in E}L_2(a_e,b_e).
\]

\begin{remark}
\label{rem:Sobolev}
By Sobolev's lemma there exists a
continuous linear operator $\psi\colon W_2^1(0,l) \to
\K^2,\
f\mapsto (f(0),f(l))$ from the first order Sobolev space to the space of
boundary values of an interval.
(In fact, one can compute that $\lVert \psi \rVert =
\bigl(\frac{\cosh l+1}{\sinh l}\bigr)^{1/2}$.)
\end{remark}

Let $E':= (E_\li\times\{0\}) \cup
(E_{\re}\times\{1\})$. 
The set $E'$ encodes all finite boundary points of all edges.
Note that \eqref{eq:edgelength} and Remark \ref{rem:Sobolev}
give rise to continuous linear mappings $\trace,\strace\from \bigoplus_{e\in
E}W_2^1(a_e,b_e) \to \ell_2(E')$,
the \emph{trace} and \emph{signed trace}, respectively,
defined by
\begin{align*}
	(\trace f)(e,j) & := \begin{cases}
			f_e(a_e) & (e\in E_\li, j=0),\\
			f_e(b_e) & (e\in E_\re, j=1),
			\end{cases}\\
	(\strace f)(e,j) & := \begin{cases}
			f_e(a_e) & (e\in E_\li, j=0),\\
			-f_e(b_e) & (e\in E_\re, j=1).
			\end{cases}
\end{align*}
The trace mappings defined above will be used in the study of the 
Laplacian on the graph.
For the 
case of 
the derivative operator we need the mappings $\trace_\li \from \bigoplus_{e\in
E}W_2^1(a_e,b_e) \to \ell_2(E_\li)$ and $\trace_\re \from \bigoplus_{e\in E}
W_2^1(a_e,b_e)\to \ell_2(E_\re)$,
defined by
\[
   (\trace_\li f)(e) := f_e(a_e) \quad(e\in E_\li),\qquad
(\trace_\re f)(e):= f_e(b_e)\quad (e\in E_\re).
\]

\subsection{The Laplace operator}

As a first application of the considerations in the previous sections we now
treat the 
Laplacian 
in $\cHg$. Define
the maximal operator $\hat H$ in $\cHg$,
\begin{align*}
	D(\hat{H}) & := \bigoplus_{e\in E} W_2^2(a_e,b_e),\\
	\hat{H}f & := (-f_e'')_{e\in E}.
\end{align*}
The operator $\hat{H}$ is the adjoint of the minimal operator defined as the
closure of the restriction of $\hat H$ to 
$\bigl(\prod _{e\in E}
C_c^\infty(a_e,b_e)\bigr)\cap D(\hat H)$.

For $f\in \bigoplus_{e\in E} W_2^1(a_e,b_e)$ 
 abbreviate $(f_e')_{e\in E}$ by $f'$.
Define $\bm\from \gr{\hat{H}}\to \ell_2(E')\oplus \ell_2(E')$
by
\[\bm(f,\hat{H}f) := (\trace f, \strace f').\]

Then $\bm$ is linear. Using \eqref{eq:edgelength} one
obtains that $\bm$ is continuous and surjective. To show surjectivity, let
$\eta\in C^\infty(0,l)$, with support $\spt\eta\sse(0,l/2)$, $\eta=1$ in a
neighbourhood of $0$. For prescribed boundary values $\alpha,\beta\in\K$ for the
function and its derivative define $f(\xi):=(\alpha+\beta \xi)\eta(\xi)$
($\xi\in(0,l)$). Then there exists a constant $c\ge0$ (independent of
$\alpha,\beta$) such that $\|f\|_2^2+\|f''\|_2^2\le c(|\alpha|^2+|\beta|^2)$.
As a consequence one obtains that for each $(\alpha,\beta)\in\ell_2(E')\oplus
\ell_2(E')$ there exists $f\in D(\hat H)$ such that
$\bm(f,\hat{H}f)=(\alpha,\beta)$. (See also~\cite{schubert2011,lsv2012}.)

For $f,g\in D(\hat{H})$, integration by parts
yields
\begin{align*}
&\scpr(f,\hat{H}g) - \scpr(\hat{H}f,g)\\ 
& = 
\sum_{e\in E_\li\cap E_\re}(f_e\overline{(-g_e')})\big|_{a_e}^{b_e}
+\sum_{e\in E_{\li}\setminus E_\re}f_e(a_e)\overline{g_e'(a_e)}
-\sum_{e\in E_{\re}\setminus E_\li}f_e(b_e)\overline{g_e'(b_e)} \\
& \qquad
-\sum_{e\in E_\li\cap E_\re}(-f_e'\overline{g_e})\big|_{a_e}^{b_e}
-\sum_{e\in E_{\li}\setminus E_\re}f_e'(a_e)\overline{g_e(a_e)}
+\sum_{e\in E_{\re}\setminus E_\li}f_e'(b_e)\overline{g_e(b_e)}
 \\
& = \scpr(\trace f , \strace g') - \scpr(\strace f' , \trace g)
= \bigscpr({\begin{pmatrix} 0 & -1 \\ 1 & 0 \end{pmatrix}\bm(f,\hat{H}f)} ,
{\bm(g,\hat{H}g)}).
\end{align*}

Defining $\Omega$ and $\omega$ to be the standard skew-symmetric form on
$\gr{\hat{H}}$ and $\ell_2(E')\oplus \ell_2(E')$, respectively,
we obtain
\[\Omega((f,\hat{H}f),(g,\hat{H}g)) = \omega(\bm(f,\hat{H}f),\bm(g,\hat{H}g))
\quad(f,g\in D(\hat{H})).\]

\begin{theorem}\label{s-a Lapl}
An operator $H\sse\hat H$ is self-adjoint if and only if there exist a
closed subspace $X\subseteq \ell_2(E')$ and a
	self-adjoint operator $L$ in $X$ such that
	\[D(H) = \set{f\in D(\hat{H})}{\trace f \in D(L),\, L\trace f = Q
\strace f'},\]
	where $Q\from \ell_2(E')\to X$ is the orthogonal projection.
\end{theorem}

\begin{proof}
The previous discussion shows that $(\Omega,\ell_2(E'),\ell_2(E'),F,\omega)$ is
a boundary system. Therefore the assertion follows from Theorem \ref{thm:sa_XL}.
\end{proof}

\begin{remarks}
(a)
We note that the boundary conditions leading to self-adjoint Laplacians can be
described in different ways. We refer to \cite[Theorem 5]{kuc08} for an account
of these descriptions. However, as already mentioned in the Introduction, in
the present general case the obtained self-adjoint operator may fail to be
semibounded, and therefore the form method may not be available.

(b) In the known cases where the operator $H$ in Theorem \ref{s-a Lapl} can
be constructed by the form method, the space $X$ and the (bounded) operator $L$
take a special role in the definition of the form: The space $X$ restricts the
domain of the form, whereas $L$ occurs in a term of the form itself;
cf.~\cite{kuc04,kkvw09}.
\end{remarks}

\begin{examples}
  (a)
  Let $V:=\Z$ and $E:=\Z$. We set $a(n):=n$, $b(n):=n+1$ and $\gamma_0(n):=n$, 
  $\gamma_1(n):=n+1$ for all $n\in\Z$. 
  Then $\Gamma=(V,E,a,b,\gamma_0,\gamma_1)$ is the metric graph $\Z$ with nearest neighbour edges.
  For $n\in \Z$ let $X_n:=\lin\{(1,1)\}$ and $X:=\bigoplus_{n\in\Z} X_n$. Define $L$ in $X$ by
  \[
D(L):= \set{(x_n)_{n\in\Z} \in X}{(nx_n)_{n\in\Z} \in X},\quad L(x_n):= 
(nx_n).
\]
  Then $L$ is self-adjoint but not bounded from below.
  By Theorem \ref{s-a Lapl} the opeator $H\subseteq \hat{H}$ with
  \[
D(H) = \set{f\in D(\hat{H})}{\trace f \in D(L),\, L\trace f = Q\strace f'}
\]
  is self-adjoint. It is easy to see that $f\in D(\hat{H})$ is in 
$D(H)$ if and only if
  \[
f_{n}(n) = f_{n-1}(n),\quad -f_{n}'(n) + f_{n-1}'(n) = 2nf_{n}(n) 
\quad(n\in\Z).
\]
  So, we have encoded $\delta$-type boundary conditions (see e.g.\ \cite[Section 3.2.1]{kuc04}) 
at $n$ with coupling parameter $2n$, for all $n\in\Z$. Note that since $L$ is 
unbounded, also $H$ is not bounded from below and the form method 
cannot be applied to define $H$.

  (b)
  Let $V:=\N_0$ and $E:=\N$. We set $a(n):=0$, $b(n):=n$, $\gamma_0(n):=0$, 
$\gamma_1(n):=n$ for all $n\in\N$. 
Then $\Gamma=(V,E,a,b,\gamma_0,\gamma_1)$ is a metric graph, 
and $0\in V$ is a vertex with infinite degree. Let 
$X_0:=\ell_2(\N)$ and for $n\in\N$ set $X_n:=\{0\}$. Let 
$X:=\bigoplus_{n\in\N_0} X_n$. 
  Let $L_0$ be a self-adjoint operator in $\ell_2(\N)$ which is not 
bounded from below. We set 
  \[
D(L):=\set{(x_n)_{n\in\N_0} \in X}{x_0\in D(L_0)},\quad L(x_n):= 
(L_0x_0,0,\ldots).
\]
  Then $L$ is self-adjoint in $X$ and not bounded from below. 
  By Theorem \ref{s-a Lapl}, $H\subseteq \hat{H}$ defined by $X$ and $L$ is self-adjoint, and also not bounded from below. Thus, the form method to define $H$ is not applicable.
\end{examples}

\subsection{The first derivative operator}

As another application, we describe boundary conditions for the 
derivative operator
that lead to skew-self-adjoint operators.
We define the maximal operator
$\hat H$  on $\cHg$,
\begin{align*}
  D(\hat H) &:= \bigoplus_{e\in E} W_2^1(a_e,b_e), \\
 \hat H f &:= ( f_e')_{e\in E}.
\end{align*}
The operator $\hat H$ is the negative adjoint of the minimal one defined as the
closure of the restriction of $\hat H$ to $\left(\prod_{e\in E}
C_c^\infty(a_e,b_e)\right)\cap D(\hat H)$.

Let $\bm
\colon
G(\hat H) \to  \ell_2(E_\re)\oplus \ell_2(E_\li)$,
\[
   \bm(f,\hat H f) := (\trace_\re f, \trace_\li f). 
\]
Then $F$ is linear, surjective and continuous (for surjectivity construct piecewise affine functions). 
For $f,g\in D(\hat H)$,
integration by parts yields
\begin{align*}
   \scpr( f, \hat Hg) + \scpr(\hat H f, g) & = \sum_{e\in E_\re} f_e(b_e)
\overline{g_e(b_e)} - \sum_{e\in E_\li} f_e(a_e) \overline{g_e(a_e)} \\
        & =  \omega (\bm(f, \hat H f), F(g, \hat H g)),
\end{align*}
 where $\omega$ is the standard unitary form on $\ell_2(E_\re)\oplus
\ell_2(E_\li)$;
cf.~Example \ref{ex:forms}(c).
As a consequence, 
denoting by $\Omega$
the standard symmetric form on $G(\hat H)$, we obtain
\[
 \Omega((f,\hat H f), (g, \hat H g)) = \omega (\bm(f, \hat H f), \bm(g, \hat H
g)) \quad (f,g \in D(\hat H)).
\]

\begin{theorem}
\label{thm:ssa}
An operator $H\sse\hat H$ is skew-self-adjoint if and only if there
exists a unitary operator $L$ from $\ell_2(E_\re)$ to
$\ell_2(E_\li)$ such that
\[
          D(H) = \sset{ f\in D(\hat H)}{ L \trace_\re f = \trace_\li f }.
\]
\end{theorem}

\begin{proof}
The previous discussion shows that
$(\Omega,\ell_2(E_\re),\ell_2(E_\li),F,\omega)$ is a boundary system. Therefore
the assertions follow from Theorem \ref{thm:ssa_L}.
\end{proof}

\begin{remarks}
(a)
  Theorem \ref{thm:ssa} precisely characterizes when there exist
skew-self-adjoint restrictions of $\hat H$.
  Namely, the spaces $\ell_2(E_\re)$ and $\ell_2(E_\li)$ have to be unitarily
equivalent, i.e., $E_\re$ and $E_\li$ have the same cardinality.

	(b)
	An operator $H$ is skew-self-adjoint if and only if 
it is the generator of a $C_0$-group of unitary operators.
Therefore, Theorem
\ref{thm:ssa} describes precisely the unitary groups in $\cHg$ yielding
solutions 
for the 
Cauchy problem for the
transport equation 
	\[u'(t) = Hu(t) \quad(t\in\R)\]
	on $\Gamma$.

	(c)
Let $\K=\C$.	
The Dirac operator in $\cHg$ is defined to be a self-adjoint
restriction of $-i\hat{H}$.
	Since $H$ is skew-self-adjoint if and only if $-iH$ is self-adjoint we
obtain that  
a Dirac operator
$-iH$ is self-adjoint if and only if 
	\[D(-iH) = \{ f\in D(\hat H); L \trace_\re f = \trace_\li f \}\]
	for some unitary $L$ from $\ell_2(E_\re)$ to $\ell_2(E_\li)$.
\end{remarks}

 In both
of
the situations of operators on graphs, i.e., the Laplacian and the
first order derivative, the given structure of the graphs,
encoded in the two maps $\gamma_0$ and $\gamma_1$,
did not occur in the description of the boundary conditions. 
In other words, 
we
replaced the graph by a 
so-called \emph{rose}
(cf.~\cite[Section
4.3]{kuc08}), possibly with infinitely many edges. 
If one wants to consider only `local boundary conditions', i.e., boundary
conditions respecting the adjacency relations, one will obtain the relations
and operators describing the boundary conditions in block form;
cf.~\cite[Section 5]{seivoi11}
(for the case of singular diffusion on finite graphs)
and \cite[Chapter 1]{schubert2011}.
In the case of the
self-adjoint Laplacian, each of the blocks would be a self-adjoint relation,
whereas in the case of the skew-self-adjoint 
derivative operator,
each of the blocks would be unitary.

{
}

\bigskip
\noindent
Carsten Schubert \\
Friedrich-Schiller-Universit\"at Jena \\
Fakult\"at f\"ur Mathematik und Informatik \\
Ernst-Abbe-Platz 2, 07743 Jena, Germany\\
{\tt carst@hrz.tu-chemnitz.de}\\[3ex]
Christian Seifert\\
TU Hamburg-Harburg\\
Institut f\"ur Mathematik\\
Schwarzenbergstr. 95E\\
21073 Hamburg, Germany\\
{\tt christian.seifert@tu-harburg.de}\\[3ex]
J\"urgen Voigt and Marcus Waurick\\
Technische Universit\"at Dresden\\
Fachrichtung Mathematik\\
01062 Dresden, Germany\\
{\tt juergen.voigt@tu-dresden.de}\\
{\tt marcus.waurick@tu-dresden.de}
\end{document}